\documentclass{article}
\usepackage{graphicx}
\usepackage{amsfonts}
\usepackage{amsmath}
\usepackage{amssymb}
\usepackage{url}
\usepackage{fancyhdr}
\usepackage{indentfirst}
\usepackage{enumerate}

\usepackage{times}
\usepackage{bm}
\usepackage{natbib}
\usepackage{color}
\usepackage{comment}

\usepackage[plain,noend]{algorithm2e}

\usepackage[colorlinks=true,citecolor=blue]{hyperref}
\usepackage{amsthm}

\newcommand{\E}{\mathbb{E}}

\renewcommand{\ge}{\geqslant}
\renewcommand{\le}{\leqslant}
\renewcommand{\geq}{\geqslant}
\renewcommand{\leq}{\leqslant}
\renewcommand{\epsilon}{\varepsilon}

\renewcommand{\cdots}{\dots}

\theoremstyle{plain}
\newtheorem{theorem}{Theorem}

\newtheorem{lemma}[theorem]{Lemma}
\newtheorem{proposition}[theorem]{Proposition}
\theoremstyle{definition}

\newtheorem{example}{Example}[section]

\theoremstyle{remark}

\numberwithin{equation}{section} \numberwithin{theorem}{section}

\def\T{\top}
\def\E{\mathbb{E}}
\def\var{\mathrm{var}}

\begin{document}

\title{Compatible Matrices of Spearman's Rank Correlation}

\author{Bin Wang\thanks{RCSDS, NCMIS, Academy of Mathematics and Systems Science,
Chinese Academy of Sciences, Beijing 100190, China. Email: \url{wangbin@amss.ac.cn} }
\and Ruodu Wang\thanks{Department of Statistics and Actuarial Science, University of Waterloo, Waterloo, Ontario N2L 3G1, Canada. Email: \url{wang@uwaterloo.ca}} \and
Yuming Wang\thanks{School of Mathematical Sciences, Peking University, Beijing 100871, China. Email: \url{wangyuming@pku.edu.cn}}
}
\date{}
\maketitle

\begin{abstract}
In this paper, we provide a negative answer to a long-standing open problem on the compatibility of Spearman's rho matrices. Following an equivalence of Spearman's rho matrices and linear correlation matrices for dimensions up to 9 in the literature, we show non-equivalence for dimensions 12 or higher. In particular, we connect this problem with the existence of a random vector under some linear projection restrictions in two characterization results.


\end{abstract}

\newpage


\section{Introduction}
\subsection{Origin of the question}
The use of copulas and measures of association  has been brought into many areas of statistical applications
since the mid 1990s, when dependence was mainly thought of in terms of linear
correlation (matrices).
One of the most important measures of association is the \emph{Spearman's rank correlation coefficient} (referred to as \emph{Spearman's rho}), defined as the linear correlation coefficient between the rank variables of two random variables.  Although there are various ways to compute  measures of association in
dimension $d>2$, they are still most widely used and understood in the
bivariate case $d=2$. As a consequence,
for a model of more than two dimensions, one typically relies on a matrix of the values of pairwise Spearman's rho correlations. This is certainly analogous to the use of a covariance matrix or a correlation matrix to model dependence among random variables.

One of the key open problems of bivariate \emph{Spearman's rho (rank correlation) matrices} is their \emph{compatibility}.
Below we quote the seminal paper \cite{EMS02} in the realm of Quantitative Risk Management:
\begin{center}
	\it ``That is, given an arbitrary symmetric, positive semi-definite
	matrix with unit elements on the diagonal and off-diagonal elements in the interval
	$[-1,1]$, can we necessarily find a random vector with continuous marginals for which
	this is the rank correlation matrix, or alternatively a multivariate distribution for
	which this is the linear correlation matrix of the copula?" If we estimate a rank
	correlation matrix from data, is it guaranteed that the estimate is itself a rank
	correlation matrix? A necessary condition is certainly that the estimate is a linear
	correlation matrix, but we do not know if this is sufficient."
\end{center}

In other words, the key question is \emph{whether a linear correlation matrix is necessarily a rank correlation matrix}.
This paper is dedicated to this long-standing open question.
In particular, we obtain a negative answer for the first time:
 a linear correlation matrix is not equivalent to a rank correlation matrix  for dimension $12$ or higher.
 The fact that this question has been open for a long time was partially because of the lack of tools to justify a negative answer.
We provide two related results (Theorems \ref{main_theorem} and \ref{th:2}) on this question, which lead to a counter-example for dimension $12$.

\subsection{Known results and related literature}
 In the literature, there are quite some papers addressing this question, mostly proving that the answer to the main question is positive for low dimensions.
First,  as mentioned in \cite{EMS02}, the answer to this question is positive for the case $d=2$, that is, an arbitrary $2\times 2$ linear correlation matrix is always a rank correlation matrix. This fact can be easily verified since there is only one degree of freedom in the case of $d=2$.
For $d=3$, there are several methods to obtain a positive answer: \citet[Section 4.4.6]{Kuro-cooke2006}  made use of the vine-copula method to show that all correlation matrices can be reached by elliptical copulas; \cite{hj2006} obtained an affirmative answer for the cases $d\le 4$ by analyzing marginal distributions of spherical distributions, and had some discussions on the case $d\ge 5$;  \cite{Devr2010} obtained another construction for the desired copula.
 More recently, \cite{Devr2015} showed that the answer to the above question is positive for $d\leq 9$. Their method relies on a crucial result of \cite{Ycart85} which characterized the extreme points of the set of all linear correlation matrices. 
 However, to the best of our knowledge, it remains unsolved whether the answer is always positive for any dimension, although many guessed that the statement should be false for sufficiently large dimension (see e.g.~\cite{hj2006} and \cite{Devr2015}). This is precisely the main target of this paper.

The compatibility problems of matrices are also studied for other  bivariate measures of association; see, for instance, \cite{CJ06} and \cite{EHW16} for compatibility of Bernoulli correlation and tail-dependence matrices.

\section{Compatibility of Spearman's rho matrices}
Throughout this paper, $d$ is a positive integer, and we fix an
atomless probability space  
on which all random
variables and random vectors are defined.
The $d\times d$ identity matrix is denoted by $I_d$.
For two random variables $X_1$ and $X_2$ with continuous distributions $F_1$ and $F_2$, respectively, their \emph{Spearman's\ rho} is defined as
\begin{equation*}
\rho^S(X_1,X_2)=\rho^P(F_1(X_1),F_2(X_2))=12\E(F_1(X_1)F_2(X_2))-3,
\end{equation*}
where $\rho^P$ is the Pearson's correlation coefficient.

For a $d$-dimensional random vector $X=(X_1,\ldots,X_d)$ with continuous marginal distributions, the \emph{Spearman's\ rho\ matrix} of $X$ is defined as
\begin{equation*}
R(X)=(R_{ij})_{d\times d},\ \mbox{where}\ R_{ij}=\rho^S(X_i,X_j)\quad i,j=1,\ldots,d.
\end{equation*}
A $d\times d$ matrix is called a Spearman's rho matrix if it is the Spearman's rho matrix of some $d$-dimensional random vector.
Since the Spearman's rho is  determined by the unique copula of $X$ (see \cite{J14} and \cite{MFE15}), it is sufficient to consider random vectors with all marginal distributions being uniform on $[0,1]$.

A square matrix is called \emph{standardized} if  all its diagonal entries are one.
Let $\mathcal P_d$ be the set of $d\times d$ standardized symmetric positive semi-definite matrices, and $\mathcal S_d$ be the set of $d\times d$ Spearman's rho matrices.
Clearly, $\mathcal P_d$ is also the set of all linear correlation matrices. Proposition \ref{prop:1} summarizes some simple properties of the set $\mathcal S_d$, which are straightforward to verify.

\begin{proposition}\label{prop:1}
	${\mathcal S}_d$ is a convex set, is closed with respect to matrix convergence, and ${\mathcal S}_d\subset \mathcal P_d$.
\end{proposition}

The key  question that we study is whether $\mathcal S_d=\mathcal P_d$, i.e.~the opposite direction of the inclusion in Proposition \ref{prop:1} holds.
As already mentioned above, the answer to the above question is positive for $d\le 9$ as shown by \cite{Devr2015}.
Precisely, for $d\le 9$,
	a $d\times d$ matrix is a Spearman's rho matrix if and only if it is a standardized symmetric positive semi-definite matrix.

To address the case $d \ge 10$, which is much more complicated to study,
we establish in Theorem \ref{main_theorem} an equivalent condition for a linear correlation matrix $R$ to be a rank correlation matrix, which will be useful later.
We first review the rank decomposition of a matrix.
If a $d\times d$ matrix $R\in\mathcal{P}_d$ has rank $k$, then, as a well-known result in linear algebra, there exists 
a $d\times k$ column-full-rank matrix $A \in {\mathbb R}^{d\times k}$ satisfying
	\begin{equation}\label{eq_1}
	R=AA^{\T}.
	\end{equation}
For a $d\times d$ matrix $R\in\mathcal{P}_d$, the matrix $A$ in \eqref{eq_1} is called a \emph{rank\ decomposition} of $R$, which is obviously not unique.
Since $R$ has diagonal entries $1$, every row vector of $A$ is a unit  vector in ${\mathbb R}^k$.
 We denote by $\mathcal A_{d, k}$ the set of $d\times k$ matrices $A$ 
 such that each row of $A$ is a unit vector. In other words, a matrix  $R$ of rank $k$ is in $\mathcal P_d$ if and only if there exists $A\in \mathcal A_{d, k}$ of rank $k$ such that $R=AA^\T$.



\begin{theorem}\label{main_theorem}
	For a matrix  $R\in \mathcal P_d$ of rank $k$, the following statements are equivalent.
	\begin{enumerate}[(a)]
\item $R \in {\mathcal S}_d$.

\item  For any rank decomposition $A\in \mathbb  R^{d\times k}$ of $R$, there exists a $k$-dimensional random vector  $V$ with mean vector $0$ and covariance matrix $I_k$ such that marginal distributions of $AV$ are all $\mathrm U[-\sqrt{3},\sqrt{3}]$.	
\item There exist a rank decomposition $A\in \mathbb R^{d\times k}$ of $R$  and a $k$-dimensional random vector $V$ with mean vector $0$ and covariance matrix $I_k$ such that marginal distributions of $AV$ are all $\mathrm U[-\sqrt{3},\sqrt{3}]$.
\end{enumerate}

\end{theorem}

\begin{proof}
	We prove the theorem via the route (a)$\Rightarrow$(b)$\Rightarrow$(c)$\Rightarrow$(a).
Note that (b)$\Rightarrow$(c) is trivial. For (c)$\Rightarrow$(a), note that $AV$ has Spearman's rho matrix $R$  because $\E[AV(AV)^\top ] =AI_kA^\top = R$, and hence $R\in {\mathcal S}_d$.
Below we show (a)$\Rightarrow$(b).

Without loss of generality,  assume the first $k$ row vectors of matrix $R=(R_{ij})_{d\times d}$ are linearly independent. From this assumption, we know that the first $k$ row vectors of any rank decomposition matrix $A=(a_{ij})_{d\times k}$ of $R$ are also linearly independent, which constitute a full rank square matrix  $B=(a_{ij})_{k\times k}$.
	Since $R\in {\mathcal S}_d$, there exists a random vector $Y=(Y_1,\ldots,Y_d)$ such that $Y_i\sim\mathrm{U}[-\sqrt{3},\sqrt{3}],~i=1,\ldots,d$ and the $Y$  has linear correlation matrix $R$.
Let $X= (Y_1,\ldots,Y_k)$ and $V=B^{-1}X$. Since $\E[Y]=0$ and $\var(Y_i)=1$ for $i=1,\ldots,d$, we know $\E[YY^\top]= R$ and hence
$
	\E[X X^\top]=BB^\T.
$
Therefore, $\E[V]=0$ and   $$\E[VV^{\T}]= B^{-1}\E[ XX^T](B^{-1})^{\T}= B^{-1}  B B^{\T} (B^{-1})^{\T}=I_k.$$

Note that $AV= A B^{-1} X$, and hence the first $k$ components of $AV$ are precisely those of $X$. It remains to verify the rest of the components in order to have $Y=AV$.
For $i=k+1,\ldots,d$, denote by $a_i$ the $i$-th row of $A$ as a (column) vector.
Note that $(R_{1i},\dots,R_{ki})= B a_i$. We can compute
\begin{align*}
\E\left[(Y_i- a_i^\T V)^2\right]
& =  \E[  Y_i^2  ] + \E\left[(a_i^\T V)^2  \right]-  2\E\left[Y_i  a_i^\T B^{-1} X \right]\\
&= 2 a^\T_i a _i-  2 a_i^\T B^{-1}\E\left[ X Y_i \right]\\
&= 2 a_i^\T a_i-  2 a_i^\T B^{-1} (R_{1i},\dots,R_{ki})\\
&= 2 a_i^\T a_i-  2 a_i^\T B^{-1}  B a_i   =0.
\end{align*}
Thus, $Y=AV$ almost surely, and the $k$-dimensional random vector $V$ satisfies the requirement in (b).
\end{proof}

Next we  review a crucial result on the property of extreme points of $\mathcal{P}_d$ shown by \cite{Ycart85} and \cite{GPW1990}. 
This result and Theorem \ref{main_theorem} give rise to a natural idea to build counter-examples for $d \ge 10$ as in Section \ref{sec:3}.
The logic below also reveals a route to the proof for the case $d\le 9$ as shown by \cite{Devr2015}.

\begin{lemma}[Ycart,1985; Grone et al.,1990]\label{lem_Ycart}
	There exist extreme points of rank $k$ in $\mathcal{P}_d$ if and only if $k(k+1)/2\le d$.
\end{lemma}

Since the set $\mathcal P_d$ is a convex compact subset in the vector space of symmetric $d\times d$ matrices which is identified with ${\mathbb R}^{d(d+1)/2}$,
according to the Krein-Milman Theorem, ${\mathcal P}_d$ is the convex hull of the set of all its extreme points, denoted as ${\mathcal E}_d$. Furthermore, in the light of Carath{\'e}odory's Theorem, any matrix in ${\mathcal P}_d$ can be written as a convex combination of at most $d(d+1)/2+1$ matrices in ${\mathcal E}_d$. Due to the linearity of the mapping from copulas to rank correlation matrices,
we need and only need to verify whether every matrix $R$ in ${\mathcal E}_d$ is a Spearman's rho matrix.
In case $d\le 9$, according to Lemma \ref{lem_Ycart}, all extreme points in $\mathcal E_d$ have rank at most 3.
If $d\geq 10$, the rank $k$ of a matrix $R\in{\mathcal E}_d$ can take the value $4$.
This fact is key to distinguish the cases $d\le 9 $ and $d\ge 10$ for the compatibility problem of Spearman's rho matrices.

Combining Lemma \ref{lem_Ycart} and Theorem \ref{main_theorem}, we arrive at the following equivalent condition for $\mathcal S_d=\mathcal P_d$.
Denote by $k_d=\max\{k\in \mathbb N: k(k+1)/2 \le d\}$.

\begin{theorem}\label{th:2}
$\mathcal S_d=\mathcal P_d$ if and only if for every $A\in \mathcal A_{d, k_d}$, there exists a $k_d$-dimensional random vector  $V$ with mean vector $0$ and covariance matrix $I_{k_d}$ such that marginal distributions of $AV$ are all $\mathrm U[-\sqrt{3},\sqrt{3}]$.
\end{theorem}
\begin{proof}
For   $k\in \mathbb N$ and $A\in \mathcal A_{d,k}$, denote by $\mathcal V(A)$ the set of $k$-dimensional random vectors  with mean vector $0$ and covariance matrix $I_{k}$ such that marginal distributions of $AV$ are all $\mathrm U[-\sqrt{3},\sqrt{3}]$.
\begin{enumerate}[(i)]
\item \underline{The ``if" statement.}
Suppose that $\mathcal V(A)$ is non-empty for every $A\in \mathcal A_{d, k_d}$.
As we have seen from Lemma \ref{lem_Ycart}, every $R\in \mathcal E_d$ has rank at most $k_d$.
Denote by $k$ the rank of $R$ and write $R$ as  $R=AA^\T$ for some $A \in \mathcal A_{d,{k}}$ of rank $k$, then $k\le k_d$.
If $k=k_d$, since $\mathcal V(A)$ is non-empty, by Theorem \ref{main_theorem} ``(c)$\Rightarrow$(a)", we know that $R\in \mathcal S_d$.
If $k<k_d$, then we can take $B=(b_{ij})_{d\times k_d}$ where $(b_{ij})_{d\times k}=A$ and $b_{ij}=0$ for $j >k$ and $i=1,\dots,d$.
Note that $\mathcal V(B)$ is non-empty. Take $V=(V_1,\dots,V_{k_d})\in \mathcal V(B)$.
Since  $BV=A (V_1,\dots,V_k)$, we have $(V_1,\dots,V_k)\in \mathcal V(A)$.
Therefore, $\mathcal V(A)$ is also non-empty, and by Theorem \ref{main_theorem} ``(c)$\Rightarrow$(a)", we know that $R\in \mathcal S_d$.
In both cases, $R\in \mathcal S_d$, which further implies  $\mathcal E_d\subset \mathcal S_d$. From the fact that $\mathcal S_d$ is convex (Proposition \ref{prop:1}), we have  $\mathcal S_d=\mathcal P_d$.

\item \underline{The ``only-if" statement.}
Suppose that $\mathcal S_d=\mathcal P_d$.
Take any $A\in \mathcal A_{d,k_d}$, let $k$ be the rank of $A$ and  $R=AA^\T$. Clearly $R\in \mathcal P_d = \mathcal S_d$.
If $k=k_d$,
by Theorem \ref{main_theorem} ``(a)$\Rightarrow$(b)", $\mathcal V(A)$ is non-empty.
If $k<k_d$, 
then all row vectors of $A$ belong to a $k$-dimensional subspace of $\mathbb{R}^{k_{d}}$. Thus there exists a rotation transformation in $\mathbb{R}^{k_d}$ which can be denoted as a $k_{d}\times k_d$ orthogonal matrix $Q$, such that each row of $AQ$ has the last $k_d-k$ entries  0.
Define a $d\times k$ matrix $\bar A$ from the first $k$ columns of $AQ$. It follows that $\bar A\bar A^\T= A QQ^\T   A^\T = AA^\T =R \in \mathcal S_d$.
Then, by Theorem \ref{main_theorem} ``(a)$\Rightarrow$(b)",  
there exists a $k$-dimensional random vector $\bar V\in \mathcal V(\bar A)$.
Let $\hat V$ be a $k_d$-dimensional random vector, with its first $k$ components equal to those of $\bar V$, and the rest components follow a $\mathrm N(0,I_{k_d-k})$ distribution and independent of $\bar V$.
Define $V=Q\hat V$, then $V$ has mean vector 0 and covariance matrix $I_{k_d}$.
Since all entries of $AQ$ other than the first $k$ columns are zero, we have
$ AV  =A  Q \hat V=  \bar A\bar V.$
Therefore, all margins of $A V$ are $\mathrm U[-\sqrt{3},\sqrt{3}]$. Hence $V\in \mathcal V(A)$. 
  \end{enumerate}
\end{proof}

Theorem \ref{th:2} together with the following simple result leads to the positive answer of $\mathcal S_d=\mathcal P_d$ for $d\le 9$.

\begin{lemma}[The Archimedes Theorem]\label{lem:3}
Let $V$ be uniformly distributed over the  unit sphere in $\mathbb R^3$. For all $d\in \mathbb N$ and $A\in \mathcal A_{d,3}$, the marginal distributions of $AV$ are all  $\mathrm U[-1,1]$.
\end{lemma}

Lemma \ref{lem:3} follows from the simple fact that the uniform distribution on the three dimensional unit sphere has $\mathrm{U}[-1,1]$ marginal distributions, and since it is invariant under rotations, it has uniform projections in all directions.
This statement is not true for dimensions other than $3$.

For the case $d \ge 10$, 
we have $k_d\geq 4$ and it is unclear whether for every $A\in\mathcal A_{d,k_d}$, the random vector $V$ still exists satisfying the condition in Theorem \ref{th:2}. In the next section, we give a counter-example for dimension 12.
Using Theorem \ref{th:2}, this will lead to a negative answer to the main question for $d\ge 12$.


\section{A negative answer for dimension 12 or higher}\label{sec:3}

As explained above, ${\mathcal S}_d = {\mathcal P}_d$   for $d\le 9$.
It might be natural to guess that the same condition does not hold for $d\ge 10$ since there may be matrices of rank more than $3$ in $\mathcal E_d$, and the Archimedes Theorem only works in $\mathbb{R}^3$.
Thus, we hope to show ${\mathcal S}_d \neq {\mathcal P}_d$ for $d\ge 10$.
Based on Theorem \ref{th:2},   it suffices to
 construct a matrix $A\in \mathcal A_{d, 4}$, such that there does not exist a $4$-dimensional random vector $V$ satisfying all margins of $AV$ are uniform on $[-\sqrt 3,\sqrt 3]$.
 Below we give an explicit construction of such a matrix for $d=12$.

\begin{example} We specify 12 unit   vectors $a_1,\ldots,a_{12}$ in ${\mathbb R}^4$ as
\begin{equation*}\arraycolsep=2pt\def\arraystretch{1.4}
\begin{array}{llll}
a_1=(1,0,0,0), & a_2=(0,1,0,0),& a_3=(0,0,1,0),& a_4=(0,0,0,1),\\
a_5=\frac12\left(1,1,1,1\right),& a_6=\frac12\left(1,1,1,-1\right),& a_7=\frac12\left(1,1,-1,1\right),& a_8=\frac12\left(1,1,-1,-1\right),\\
a_9=\frac12\left(1,-1,1,1\right),& a_{10}=\frac12\left(1,-1,1,-1\right),& a_{11}=\frac12\left(1,-1,-1,1\right),& a_{12}=\frac12\left(1,-1,-1,-1\right).
\end{array}
\end{equation*}
We assert that there does not exist any random vector $V=(V_1,V_2,V_3,V_4)$ such that
$a_k^\T V\sim\mathrm U[-\sqrt{3},\sqrt{3}]$ for all $k=1,\ldots,12$. If our assertion is true,  Theorem  \ref{th:2} leads to
${\mathcal S}_{12}\neq {\mathcal P}_{12}$.
More precisely, using Theorem \ref{main_theorem},
 the $12\times 12$ linear correlation matrix $M$ defined by
$
M  = (a_i^\T a_j)_{12\times 12}
$
is not a Spearman's rho matrix, i.e. $M\in{\mathcal P}_{12}$ and $ M\notin {\mathcal S}_{12}$.  This matrix $M$ has rank 4 and we provide its explicit form, where $\alpha=1/2$:
\begin{equation*}
M=
\left(
\begin{array}{cccccccccccc}
1 & 0 & 0 & 0 & \alpha & \alpha & \alpha & \alpha & \alpha & \alpha & \alpha & \alpha\\
0 & 1 & 0 & 0 & \alpha & \alpha & \alpha & \alpha & -\alpha & -\alpha & -\alpha & -\alpha\\
0 & 0 & 1 & 0 & \alpha & \alpha & -\alpha & -\alpha & \alpha & \alpha & -\alpha & -\alpha\\
0 & 0 & 0 & 1 & \alpha & -\alpha & \alpha & -\alpha & \alpha & -\alpha & \alpha & -\alpha\\
\alpha & \alpha & \alpha & \alpha & 1 & \alpha & \alpha & 0 & \alpha & 0 & 0 & -\alpha\\
\alpha & \alpha & \alpha & -\alpha & \alpha & 1 & 0 & \alpha & 0 & \alpha & -\alpha & 0\\
\alpha & \alpha & -\alpha & \alpha & \alpha & 0 & 1 & \alpha & 0 & -\alpha & \alpha & 0\\
\alpha & \alpha & -\alpha & -\alpha & 0 & \alpha & \alpha & 1 & -\alpha & 0 & 0 & \alpha\\
\alpha & -\alpha & \alpha & \alpha & \alpha & 0 & 0 & -\alpha & 1 & \alpha & \alpha & 0\\
\alpha & -\alpha & \alpha & -\alpha & 0 & \alpha & -\alpha & 0 & \alpha & 1 & 0 & \alpha\\
\alpha & -\alpha & -\alpha & \alpha & 0 & -\alpha & \alpha & 0 & \alpha & 0 & 1 & \alpha\\
\alpha & -\alpha & -\alpha & -\alpha & -\alpha & 0 & 0 & \alpha & 0 & \alpha & \alpha & 1\\
\end{array}
\right).
\end{equation*}

Next we need to prove the above assertion. For the purpose of contradiction, we assume there exists a random vector $V=(V_1,V_2,V_3,V_4)$ such that $a_k^\T V\sim \mathrm U[-\sqrt{3},\sqrt{3}], ~k=1,\ldots,12$.
Due to our specific choices of $a_1,\ldots,a_{12}$, via some algebraic calculation cancelling all terms with odd powers, we obtain
\begin{equation}
\left\{
\begin{aligned}
(a_1^\T V)^2+\cdots+(a_{12}^\T V)^2&=3 \left\|V\right\|^2,\\
(a_1^\T V)^4+\cdots+(a_{12}^\T V)^4&=\frac{3}{2} \left\|V\right\|^4,
\end{aligned}
\right. \label{eq:keycons}
\end{equation}
where $\left\|V\right\|^2=V_1^2+V_2^2+V_3^2+V_4^2$. Taking expectations on both sides of the above two equations and
combining with $a_k^{\T} V\sim \mathrm U[-\sqrt{3},\sqrt{3}],~k=1,\ldots,12$,
we obtain
\begin{equation}
\left\{
\begin{aligned}
12\times 1=&\sum_{k=1}^{12}\E[(a_k^\T V)^2]=\E\left[3 \left\|V\right\|^2\right],\\
12 \times \frac{9}{5}=&\sum_{k=1}^{12}\E[(a_k^\T V)^4]=\E\left[\frac{3}{2} \left\|V\right\|^4\right],
\end{aligned}
\right. \label{eq:keycons2}
\end{equation}
which leads to
\begin{equation*}
\E[\left\|V\right\|^2]=4~~\mbox{and}~~\E[\left\|V\right\|^4]=\frac{72}{5}.
\end{equation*}
Therefore, we have $(\E[\left\|V\right\|^2])^2 = 16 > 72/5 =  \E[\left\|V\right\|^4],$ which contradicts the H{\"o}lder inequality.
Hence, no such random vector $V$ may exist.\label{ex:1}
\end{example}

We summarize the finding in Example \ref{ex:1} and the case $d\ge 12$ in the following proposition.

\begin{proposition}\label{main_proposition}
	For $d\geq 12$, ${\mathcal S}_d \neq {\mathcal P}_d$.
\end{proposition}
\begin{proof}[Proof of Proposition~\ref{main_proposition}]
The case $d=12$ is precisely addressed by Example \ref{ex:1}. For $d>12$, one can simply take the $d\times d$ matrix
\begin{equation*}
M_d=
\left(
\begin{array}{cc }
M & 0  \\ 0 & I_{d-12}
\end{array}
\right).
\end{equation*}
It is obvious that $M_d\in \mathcal P_d$ since $M\in \mathcal P_d$.
If a $d$-dimensional random vector  has $M_d$ as its Spearman's rho matrix,
its first 12-dimensional marginal random vector would have Spearman's rho matrix $M$.
As  this is not possible, we know $\mathcal S_d\ne \mathcal P_d$ for $d\ge 12$.

\end{proof}

Based on Example \ref{ex:1}, we also obtain a sharp contrast to Lemma \ref{lem:3}, summarized in the following proposition.

\begin{proposition}\label{prop:2}
For any $d\in \mathbb N$, there does not exist a $4$-dimensional random vector $V$ such that for all $A\in \mathcal A_{d,4}$, the marginal distributions of $AV$ are all $\mathrm U[-1,1]$.
\end{proposition}

If  the marginal distributions of $AV$ are all $\mathrm U[-1,1]$ for all $A\in \mathcal A_{d,4}$, then $V$ is necessarily spherically distributed. 
Therefore, Proposition \ref{prop:2} is also obtained from the fact that the uniform distribution is not in the class of margins of $4$-dimensional spherical distributions (see Section 4.9 of \cite{J97}).

We remark that Proposition \ref{prop:2} on its own does not imply a negative answer to the question of whether $\mathcal S_d=\mathcal P_d$ for $d\ge 10$.
According to Theorem \ref{th:2}, in order for  $\mathcal S_d=\mathcal P_d$, $10\le d\le 14$, it suffices if there exists such $V$ for each $A\in \mathcal A_{d,4}$. We do not quire the existence of such a random vector $V$ that works for all $A$, although this is possible if $V$ is $3$-dimensional and $A$ is chosen from $\mathcal A_{d,3}$.
 Therefore, a separate investigation for each case of $d=10,11,\dots,$ is required to reach the conclusion that $\mathcal S_d\ne \mathcal P_d$.

\section{Discussion}

Our argument for dimension $d=12$ relies very precisely on the symmetric choice of the vectors $a_1,\ldots,a_{12}$ in Example \ref{ex:1}.
The 12 vectors are chosen such that the
24 vectors $\pm a_1,\ldots,\pm a_{12}$  approximately uniformly spread out over the unit sphere in $\mathbb R^4$.
This construction is essential to arrive to equations \eqref{eq:keycons}-\eqref{eq:keycons2}.
In  \eqref{eq:keycons2}, both the left-hand side and the right-hand side can be written as   moments of a single random variable $||V||$ without explicitly relying on its components or its dependence structure.
This step finally leads to the desired  contradiction $(\E[\left\|V\right\|^2])^2  >    \E[\left\|V\right\|^4].$

However, in dimension $d=10$ or $11$, such symmetry does not exist.  After many attempts, we were not able to find a way to construct   vectors $a_1,\ldots,a_{d}$
so that an equation similar to \eqref{eq:keycons}-\eqref{eq:keycons2} can be found for these cases. Recall that in  \eqref{eq:keycons}-\eqref{eq:keycons2}, we need to cancel all terms with odd powers to arrive to the desirable conclusion, which is  very restrictive on the symmetry of $a_1,\ldots,a_d$.

In view of Proposition \ref{prop:2}, we would naturally guess  ${\mathcal S}_d\neq{\mathcal P}_d$ for $d=10$ and $d=11$. We are not aware of a counter-example in these cases.

  In the case $d\le 9$,
although we know $\mathcal S_d=\mathcal P_d$,
that is, for a given $R\in\mathcal P_d$, there exists a random vector $Y$ that has Spearman's rho matrix $R$,
it is yet unclear how to  find such $Y$, its corresponding copula model, or how to simulate from this model.
On the other hand, for any linear correlation matrix $R\in \mathcal P_d$ of rank at most $3$, one can directly use a decomposition $R=AA^\T$ for some $A\in \mathcal A_{d,3}$ 
and a uniform random vector $V$ on the unit sphere in $\mathbb R^3$. By Lemma \ref{lem:3} and simple calculation, 
$AV$ has the rank correlation matrix $R$. 
If $R$ is a convex combination of matrices $R_i$, $i=1,\dots,n$, each with rank at most $3$, 
that is, \begin{equation}\label{eq:conv} R=\sum_{i=1}^n a_i R_i \end{equation} 
for some $a_1,\dots,a_n\ge 0$, $\sum_{i=1}^n a_i=1$, 
then a copula model with rank correlation matrix $R$ can be easily constructed by a convex combination of the copulas corresponding to $R_1,\dots,R_n$. 
For a given matrix $R$, however, it is unclear how to find the corresponding convex combination \eqref{eq:conv}, even if we know such a combination always exists for $d\le 9$.
A practical alternative method (\cite{IC82}, \cite{EMS02}) to address this problem  is to use a Gaussian copula with the correlation matrix parameter $R$ to approximate this model, and the relative error is known to be at most $(\pi-3)/\pi$ (see e.g.~Section 6.2 of \cite{EMS02});  this method does not give a model with precisely the  Spearman's rho matrix $R$.
For the case of $d\ge 10$, it seems even more difficult to construct a  copula model with a given Spearman's rho matrix $R$, even assuming such a model exists.



\section*{Acknowledgement}
 R.~Wang   acknowledges financial support from the Natural Sciences and Engineering Research Council of Canada (NSERC, RGPIN-2018-03823, RGPAS-2018-522590).
 Y.~Wang    acknowledges financial support from the National Natural Science Foundation of China (Grant No. 11671021).





\end{document}